\documentclass[12pt,leqno]{amsart}
\usepackage{amsfonts,amsthm,amsmath,mathtools,amssymb,comment,xcolor,enumitem}

\theoremstyle{plain}
\newtheorem{thm}{Theorem}[section]

\newtheorem{cor}{Corollary}[section]

\theoremstyle{definition}
\newtheorem{df}{Definition}[section]
\newtheorem{rem}{Remark}[section]

\newtheorem{conj}{Conjecture}[section]

\newcommand{\FF}{\mathbb{F}}
\newcommand{\ZZ}{\mathbb{Z}}

\newcommand{\cwe}{\mathbf{cwe}}

\newcommand{\R}{\mathfrak{R}}

\newcommand{\Jac}{\mathfrak{Jac}}
\newcommand{\J}{\mathcal{J}}

\newcommand{\D}{\mathcal{D}}

\newcommand{\II}{\mathrm{II}}

\def\bm#1{\mathbf{#1}}

\DeclareMathOperator{\supp}{supp}
\DeclareMathOperator{\comp}{comp}

\DeclareMathOperator{\wt}{wt}

\begin{document}

\title[Variants of Jacobi Polynomials]{Variants of Jacobi Polynomials in coding theory}

\author[Chakraborty]{Himadri Shekhar Chakraborty*}
\thanks{*Corresponding author}
\address
	{
		(1) Graduate School of Natural Science and Technology\\
		Kanazawa University\\  
		Ishikawa 920-1192, Japan\\
		(2) Department of Mathematics, Shahjalal University of Science and Technology\\ Sylhet-3114, Bangladesh\\
	}
\email{himadri-mat@sust.edu}

\author[Miezaki]{Tsuyoshi Miezaki}
\address
	{
		Faculty of Science and Engineering, 
		Waseda University, 
		Tokyo 169-8555, Japan\\ 
	}
\email{miezaki@waseda.jp}

\date{\today}

\begin{abstract}
	In this paper, 
	we introduce the notion of 
	the complete joint Jacobi polynomial 
	of two linear codes of length~$n$ 
	over $\FF_{q}$ and $\ZZ_{k}$. 
	We give the MacWilliams type identity for 
	the complete joint Jacobi polynomials of codes.
	We also introduce 
	the concepts of the average Jacobi polynomial 
	and the average complete joint Jacobi polynomial
	over $\FF_{q}$ and $\ZZ_{k}$.
	We give a representation of 
	the average of 
	the complete joint Jacobi polynomials 
	of two linear codes of length~$n$ 
	over $\FF_{q}$ and $\ZZ_{k}$ 
	in terms of the compositions of~$n$
	and its distribution in the codes.
	Further we present a generalization of the representation 
	for the average of the $(g+1)$-fold complete joint Jacobi polynomials of codes
	over~$\FF_{q}$ and $\ZZ_{k}$.
	Finally, we give the notion of 
	the average Jacobi intersection number
	of two codes.  
\end{abstract}


\maketitle

\section{Introduction}\label{SecIntroduction}

In~\cite{MMC1972}, MacWilliams, Mallows and~Conway provided 
the notion of the joint weight enumerator of two 
$\FF_{q}$-linear codes.
The joint weight enumerators of codes were studied by 
Dougherty, Harada and~Oura~\cite{DHO} 
over the finite ring $\ZZ_{k}$ of integers modulo ~$k ~(k \geq 2)$.
Furthermore, the average of the joint weight enumerators of 
two binary codes were introduced by Yoshida~\cite{Y1989}.
Consecutively, Yoshida~\cite{Y1991} defined 
the average intersection number of two binary codes.
In~\cite{CMxxxx}, Chakraborty and~Miezaki studied
the average of the complete joint weight enumerators
of codes over $\FF_{q}$ and $\ZZ_{k}$
and defined the
average complete joint weight enumerators
as a generalization of the average joint weight enumerators in~\cite{Y1989}.
They also studied the average intersection number of two 
codes over $\FF_{q}$ and $\ZZ_{k}$.
The notion of the Jacobi polynomial of a code over $\FF_{q}$ 
was introduced by Ozeki~\cite{Ozeki1997}. 
A successive study of the Jacobi polynomial of a 
binary code was carried out by 
Bonnecaze, Mourrain and Sol\'e~\cite{BMS1999}
to construct various kinds of designs. 
The concept of the $g$-th Jacobi polynomial 
of a binary code was given by Honma, Okabe and~Oura~\cite{HOO2020}.
They also obtained the MacWilliams identity 
for the $g$-th Jacobi polynomials.
In the present paper, 
we give the notion of the complete joint Jacobi polynomials
of codes over $\FF_{q}$ and $\ZZ_{k}$ and 
obtain the MacWilliams type identity for the polynomials.
We also define the average Jacobi polynomial and 
the average complete joint Jacobi polynomial of codes 
over $\FF_{q}$ and $\ZZ_{k}$, 
and give an analogue of 
the main theorem
in~\cite{Y1989}
for each of the polynomial.
Moreover, as a generalization of the complete joint Jacobi polynomials 
and the average complete joint Jacobi polynomials,
we introduce the concept of
the~$g$-fold complete joint Jacobi polynomials and 
the average $(g+1)$-fold complete joint Jacobi polynomials,
respectively
of codes over~$\FF_{q}$ and~$\ZZ_{k}$.
Finally, we define the average Jacobi intersection number
and give a formula to compute the numbers. 
We also give some numerical examples of the 
average Jacobi intersection number for Type~$\II$ codes. 

Throughout this paper, 
we assume that $\mathfrak{R}$ denotes 
either the finite field $\FF_{q}$ of order $q$, 
where $q$ is a prime power 
or the ring $\ZZ_{k}$ of integers modulo~$k$ 
for some positive integer $k \ge 2$.

This paper is organized as follows.
In Section~\ref{Sec:Preli},
we give definitions and some basic properties
of linear codes 
over~$\R$. 
In Section~\ref{Sec:MacWilliams},
we give the MacWilliams type identity (Theorem~\ref{ThMacWilliams}) 
for the complete joint Jacobi polynomials of codes over~$\R$.
In Section~\ref{Sec:MainResults},
we give an analogue to the main theorem in~\cite{Y1989} 
for the average Jacobi polynomials (Theorem~\ref{ThAverageJacobi})
as well as for the average complete joint Jacobi polynomials (Theorem~\ref{ThAvJointJacobi}) 
over~$\R$.
In Section~\ref{Sec:AvHigherGenJacobi},
we give a generalization of the MacWilliams identity 
(Theorem~\ref{Thm:GenMacWilliams})
for the $g$-fold complete joint Jacobi polynomials of
codes over~$\R$. 
We also give a generalization of Theorem~\ref{ThAvJointJacobi}
for the average $(g+1)$-fold complete joint Jacobi polynomials
of codes over~$\R$ (Theorem~\ref{Thm:GenJacYoshida}).
In Section~\ref{Sec:AvJacobiInterNum},
we define the average Jacobi intersection number
and give a formula (Theorem~\ref{ThAvJacobiInterNum}) 
to compute this number. 
We also give some numerical examples
for some Type~$\II$ codes over $\FF_{2}$.
From the observation of the numerical examples, 
we enclose the section with 
two conjectures 
(Conjecture~\ref{ConjAsym}, Conjecture~\ref{Conjdesign}).

\section{Preliminaries}\label{Sec:Preli}

An $\FF_q$-linear code of length~$n$ is a vector subspace of $\FF_{q}^{n}$,
and a $\ZZ_{k}$-linear code of length~$n$ is an additive group of $\ZZ_{k}^{n}$. 
The elements of an~$\R$-linear code are called \emph{codewords}. 
Let $\bm{u} = (u_1,u_2,\dots, u_n)$ 
and $\bm{v} = (v_1,v_2,\dots,v_n)$ 
be the elements of $\R^{n}$. 
Then the \emph{inner product} of two elements 
$\bm{u},\bm{v} \in \R^{n}$ 
is given by
\[
	\bm{u} \cdot \bm{v} 
	:= 
	u_{1} v_{1} + u_{2} v_{2} + \dots + u_{n}v_{n}.
\]
The \emph{dual code} of an $\R$-linear code~$C$ 
of length~$n$ is defined by
\[
	C^\perp 
	:= 
	\{\bm{v}\in \R^{n} 
	\mid \bm{u} \cdot \bm{v} = \bm{0} \text{ for all } \bm{u}\in C\}. 
\]
An $\R$-linear code $C$ is called \emph{self-dual}
if $C = C^\perp$.
The \emph{weight} of $\bm{u} \in \R^{n}$ is denoted 
by~$\wt(\bm{u})$ and
defined to be the number of $i$'s such that $u_{i} \neq 0$. 
A self-dual code over $\FF_{2}$ of 
length~$n\equiv 0 \pmod 8$ is called \emph{Type~$\II$}
if the weight of each codeword of the code is a multiple of~$4$.


Let the elements of $\R$ be
$0=\omega_{0},\omega_{1},\ldots,\omega_{|\R|-1}$
in some fixed order.
Then the \emph{composition} of an element $\bm{u}\in \R^{n}$ 
is defined as
\[
	\comp(\bm{u}) 
	:= \ell(\bm{u}) 
	:= (\ell_{a}(\bm{u}) : a \in \R),
\]
where 
$\ell_{a}(\bm{u})$
denotes 
the number of coordinates of $\bm{u}$ that are equal to 
$a \in \R$.
Obviously,
\[
	\sum_{a \in \R}
	\ell_{a}(\bm{u}) = n.
\]
In general, 
a \emph{composition} $L$ of $n$ is a vector 
with non-negative integer components 
$L_{a}$ for $a \in \R$
such that
\[
	\sum_{a \in \R}
	L_{a} 
	= 
	n.
\]

\begin{df}
	Let $C$ be an $\R$-linear code of length~$n$.
	We denote
	\[
		T_{L}^{C}
		:=
		\{
			\bm{u} \in C \mid 
			\comp(\bm{u}) 
			= 
			L
		\}.
	\]
	Then the \emph{complete weight enumerator} for $C$ 
	is defined as
	\[
		\cwe_{C}(\{x_{a}\}_{a \in \R})
		:=
		\sum_{\bm{u} \in C}
		\prod_{a \in \R}
		x_{a}^{\ell_{a}(\bm{u})}
		=
		\sum_{L}
		A_{L}^{C}
		\prod_{a \in \R}
		x_{a}^{L_{a}},
	\]
	where $A_{L}^{C} = |T_{L}^{C}|$.
	In more general, the complete weight enumerator 
	of~$C$ for genus $g$
	is defined as
	\[
		\cwe_{C}^{(g)}(\{x_{a}\}_{a \in \R^{g}})
		:=
		\sum_{\bm{u}_{1},\ldots,\bm{u}_{g} \in C}
		\prod_{a \in \R^{g}}
		x_{a}^{n_{a}(\bm{u}_{1},\ldots,\bm{u}_{g})}
	\]
	where
	$n_{a}(\bm{u}_{1},\ldots,\bm{u}_{g})$ 
	denotes the number of $i$ 
	such that $ a = (u_{1i},\ldots,u_{gi})$.
\end{df}

Now fix $\bm{w} \in \R^{n}$.
We denote by
$\comp_{\bm{w}}(\bm{u}):= r(\bm{u};\bm{w})$
the \emph{Jacobi composition} 
of~$\bm{u}\in \R^{n}$ 
with respect to $\bm{w}$
having the components
$r_{a}(\bm{u};\bm{w})$
for 
$a \in \R^{2}$
which 
are defined as
\[
	r_{a}(\bm{u};\bm{w}) 
	:= 
	\#\{i \mid (u_{i},w_{i}) = a\}.
\]
Clearly
\[
	 \sum_{a\in\R^{2}} 
	 r_{a}(\bm{u};\bm{w}) 
	 = n.
\]
In general, a \emph{Jacobi composition} $R$ of $n$ is a vector 
with non-negative integer components 
$R_{a}$ for $a \in \R^{2}$
such that
\[
	\sum_{a \in \R^{2}} R_{a} = n.
\]

\begin{df}
	Let $C$ be an $\R$-linear code of length~$n$ 
	and $\bm{w} \in \R^{n}$.
	We denote
	\[
		T_{R}^{C,\bm{w}}
		:=
		\{
			\bm{u} \in C 
			\mid 
			\comp_{\bm{w}}(\bm{u})
			= 
			R
		\}.
	\]
	Then the \emph{complete Jacobi polynomial} of $C$ 
	with respect to $\bm{w} \in \R^{n}$ 
	is defined as
	\[
		Jac(C,\bm{w};\{x_{a}\}_{a \in \R^{2}})
		:=
		\sum_{\bm{u} \in C}
		\prod_{a \in \R^{2}}
		x_{a}^{r_{a}(\bm{u};\bm{w})}
		=
		\sum_{R}
		B_{R}^{C,\bm{w}}
		\prod_{a \in \R^{2}}
		x_{a}^{R_{a}},
	\]
	where $B_{R}^{C,\bm{w}} = |T_{R}^{C,\bm{w}}|$.
\end{df}

\begin{df}
	Let $C$ and $D$ be two $\R$-linear codes of length~$n$. 
	Then the \emph{complete joint weight enumerator}
	of $C$ and $D$ is defined in~\cite{CMxxxx} as
	\[
		\J_{C,D}(\{x_{a}\}_{a \in \R^{2}})
		:=
		\sum_{\bm{u} \in C, \bm{v} \in D}
		\prod_{a \in \R^{2}}
		x_{a}^{n_{a}(\bm{u},\bm{v})},
	\]
	where $n_{a}(\bm{u},\bm{v}) := \#\{i \mid (u_{i},v_{i}) = a\}$.
\end{df}

\begin{rem}\label{Rem:JointWt}
	$\J_{C,D}(\{x_{a}\}_{a \in \R^{2}})
	=
	\sum_{\bm{v} \in D} 
	Jac(C,\bm{v};\{x_{a}\}_{a \in \R^{2}})$.
\end{rem}

\section{Complete Joint Jacobi Polynomial and MacWilliams Identity}
\label{Sec:MacWilliams}

Let us fix~$\bm{w} \in \R^{n}$. 
Then we denote by 
\[
	\comp_{\bm{w}}(\bm{u},\bm{v}) 
	:= 
	h(\bm{u},\bm{v};\bm{w})
	:=
	(h_{a}(\bm{u},\bm{v};\bm{w}) : a \in \R^{3})
\]
the \emph{joint Jacobi composition} of 
$\bm{u},\bm{v}\in \R^{n}$ 
with respect to $\bm{w}$,
where 
$
	h_{a}(\bm{u},\bm{v};\bm{w}) 
	:= 
	\#\{i \mid (u_{i},v_{i},w_{i}) = a \}.
$
Clearly
	\[
		\sum_{a \in \R^{3}}
		h_{a}(\bm{u},\bm{v};\bm{w}) 
		= n.
	\]
In general, a \emph{joint Jacobi composition} 
$H$ of $n$ denotes a vector with non-negative integer components
$H_{a}$ for $a \in \R^{3}$
such that
\[
	\sum_{a \in \R^{3}}
	H_{a} = n.
\] 

\begin{df}
	Let $C$ and $D$ be two $\R$-linear code of length~$n$. 
	Then the \emph{complete joint Jacobi polynomial} 
	of $C$ and $D$ 
	with respect to $\bm{w} \in \R^{n}$ 
	is denoted by
	$\Jac(C,D,\bm{w}; \{x_{a}\}_{a \in \R^{3}})$ and 	
	defined as
	\begin{align*}
		\Jac(C,D,\bm{w}; \{x_{a}\}_{a \in \R^{3}})
		:= 
		& \sum_{\bm{u}\in C, \bm{v}\in D} 
		\prod_{a \in \R^{3}}
		x_{a}^{h_{a}(\bm{u},\bm{v};\bm{w})}\\
		= 
		& \sum_{H}
		B_{H}^{C,D,\bm{w}} 
		\prod_{a \in \R^{3}}
		x_{a}^{H_{a}}.
	\end{align*}
	where 
	$B_{H}^{C,D,\bm{w}} 
	:= 
	\# 
	\{(\bm{u}, \bm{v}) \in C \times D 
	\mid 
	\comp_{\bm{w}}(\bm{u},\bm{v}) = H\}$.
\end{df}

\begin{rem}\label{Rem:JointJacobi} 
	Let $C$ and $D$ be two $\R$-linear code of length~$n$, 
	and $\bm{w} \in \R^{n}$. 
	Then
	we have 
	\begin{enumerate}
		\item 
		If $C = \{(0,0,\ldots,0)\}$,
		then $\Jac(C,D,\bm{w}) = Jac(D,\bm{w})$.
		
		\item 
		If $D = \{(0,0,\ldots,0)\}$, 
		then $\Jac(C,D,\bm{w}) = Jac(C,\bm{w})$.
		
		\item If $C = D$ and $\bm{w} = (0,0,\ldots,0)$,
		then $\Jac(C,D,\bm{w}) = \cwe_{C}^{(2)}$.
	\end{enumerate}
\end{rem}

In this section, we give the MacWilliams type identity for
the complete joint Jacobi polynomial of codes over~$\R$.
We review~\cite{Dougherty2017,DHO,MMC1972}
to introduce some fixed characters over~$\R$.

Let $\R = \FF_{q}$, 
where $q =  p^{f}$ for some prime number $p$.  
A \emph{character} $\chi$ of $\FF_{q}$ 
is a homomorphism from the additive group $\FF_{q}$ 
to the multiplicative group of non-zero complex numbers.
Now let $F(x)$ be a primitive irreducible polynomial 
of degree $f$ over $\FF_{p}$ 
and let $\lambda$ be a root of $F(x)$.
Then any element $\alpha \in \FF_{q}$ 
has a unique representation as:
\begin{equation}\label{EquAlpha}
	\alpha 
	=
	\alpha_{0} + \alpha_{1} \lambda	
	+ \alpha_{2} \lambda^{2}
	+ \cdots +
	\alpha_{f-1} \lambda^{f-1},
\end{equation} 
where $\alpha_{i} \in \FF_{p}$.
We define $\chi(\alpha) := \zeta_{p}^{\alpha_{0}}$, 
where $\zeta_{p}$ is the $p$-th primitive root of unity,
and $\alpha_{0}$ is given by Equation~(\ref{EquAlpha}).

Again if $\R = \ZZ_{k}$,
then for
$\alpha\in\ZZ_{k}$,
we define 
$\chi$ as $\chi(\alpha) := \zeta_{k}^{\alpha}$,
where
$\zeta_{k}$ is the $k$-th primitive root of unity.

Now for any $\alpha \in \R$, 
we have the following property:
\[
	\sum_{i = 0}^{|\R|-1}
	\chi(\alpha\omega_{i}) 
	:= 
	\begin{cases}
		|\R| & \mbox{if} \quad \alpha = 0, \\
		0 & \mbox{if} \quad \alpha \neq 0.
	\end{cases} 
\]

Now we give the MacWilliams relation
for the complete joint Jacobi polynomial
of codes over $\R$.

\begin{thm}[MacWilliams Identity]\label{ThMacWilliams}
	Let $C$ and $D$ be two 
	$\R$-linear codes of length $n$, 
	and $\Jac(C,D,\bm{w}; \{x_{a}\}_{a \in \R^{3}})$ 
	be a complete joint Jacobi polynomial 
	for codes $C$ and $D$ 
	with respect to $\bm{w}\in \R^{n}$. 
	Then
	\begin{align*}
		\Jac (C,D^{\perp},&\bm{w}; \{x_{a}\}_{a \in \R^{3}})\\
		= &
		\dfrac{1}{|D|} 
		\Jac
		\left(
			C,D,\bm{w}; 
			\left\{
				\sum_{b \in \R}
				\chi(a_2 b) x_{(a_1 b a_3)}
			\right\}_{a \in \R^{3}}
		\right).
	\end{align*}
\end{thm}

\begin{proof}
	 Let
	\[
		\delta_{D^\perp}(v) := 
		\begin{cases}
			1 & \mbox{if} \quad v \in D^\perp, \\
			0 & \mbox{otherwise}.
		\end{cases} 
	\]
	Then we have the following identity
	\[
		\delta_{D^\perp}(v) 
		= 
		\dfrac{1}{|D|} \sum_{d \in D} \chi(d \cdot v).
	\]
	Now
	\begin{align*}
		\Jac(C,& {D}^\perp,\bm{w}; \{x_{a}\}_{a \in \R^{3}}) \\
		& = 
		\sum_{\bm{c} \in C} 
		\sum_{\bm{d}^\prime \in {D}^\perp} 
		\prod_{a \in \R^{3}} 
		x_{a}^{h_{a}(\bm{c},\bm{d}^\prime;\bm{w})} \\
		& = 
		\sum_{\bm{c} \in C} 
		\sum_{\bm{v} \in \R^n} 
		\delta_{D^\perp}(\bm{v}) 
		\prod_{a \in \R^{3}} 
		x_{a}^{h_{a}(\bm{c},\bm{v};\bm{w})} \\
		& =
		\dfrac{1}{|D|} 
		\sum_{\bm{c} \in C} 
		\sum_{\bm{v} \in \R^{n}} 
		\sum_{\bm{d} \in D} 
		\chi(\bm{d} \cdot \bm{v}) 
		\prod_{a \in \R^{3}} 
		x_{a}^{h_{a}(\bm{c},\bm{v};\bm{w})} \\
		& = 
		\dfrac{1}{|D|}
		\sum_{\substack{\bm{c} \in C\\ \bm{d} \in D}} 
		\sum_{(v_1,\dots,v_n) \in \R^{n}} 
		\chi(d_1v_1 + \dots + d_nv_n) 
		\prod_{1\leq i \leq n} 
		x_{(c_i v_i w_i)} 
		\\
		& = 
		\dfrac{1}{|D|}
		\sum_{\substack{\bm{c} \in C\\ \bm{d} \in D}} 
		\prod_{1\leq i \leq n} 
		\sum_{v_i \in \R} 
		\chi(d_iv_i) 
		x_{(c_i v_i w_i)} 
		\\
		& = 
		\dfrac{1}{|D|}
		\sum_{\substack{\bm{c} \in C\\ \bm{d} \in D}} 
		\prod_{a = (a_1,a_2,a_3) \in \R^{3}} 
		\left(
		\sum_{b \in \R}
		\chi(a_2 b) 
		y_{(a_1 b a_3)}
		\right)^{h_{a}(\bm{c},\bm{d};\bm{w})} \\
		& = 
		\dfrac{1}{|D|}
		\Jac
		\left(
			C,D,\bm{w};
			\left\{
				\sum_{b \in \R}
				\chi(a_2 b) 
				x_{(a_1 b a_3)}
			\right\}_{a \in \R^{3}}
		\right).	
	\end{align*}
	Hence the proof is completed.
\end{proof}

\begin{cor}
	Let $C$ and $D$ be two $\R$-linear codes of length $n$. 
	Then
	\begin{itemize}
		\item [(i)]
		$\begin{aligned}[t]
			\Jac
			(C^{\perp},D,\bm{w}; &\{x_{a}\}_{a \in \R^{3}})\\
			= &
			\dfrac{1}{|C|} 
			\Jac
			\left(
				C,D,\bm{w};
				\left\{ 
					\sum_{b \in \R}
					\chi(a_1 b) 
					x_{(b a_2 a_3)}
				\right\}_{a \in \R^{3}}
			\right).
		\end{aligned}$
		\item [(ii)]
		$\begin{aligned}[t]
			&\Jac 
			(C^{\perp},D^{\perp},\bm{w}; \{x_{a}\}_{a \in \R^{3}})\\
			& = 
			\dfrac{1}{|C||D|} 
			\Jac
			\left(
				C,D,\bm{w};
				\left\{ 
					\sum_{b_1,b_2 \in \R}
					\chi(a_1 b_1+a_2 b_2) 
					x_{(b_1 b_2 a_3)}
				\right\}_{a \in \R^{3}}
			\right).
		\end{aligned}$
	\end{itemize}
\end{cor}

Now by Remark~\ref{Rem:JointJacobi}
and by Theorem~\ref{ThMacWilliams}
we have the MacWilliams type identity for
the complete Jacobi polynomial of codes over~$\R$
as follows:
\[
	Jac(C^{\perp},\bm{w}; \{x_{a}\}_{a \in \R^{2}})\\
	=
	\dfrac{1}{|C|} 
	Jac
	\left(
		C,\bm{w};
		\left\{ 
			\sum_{b \in \R}
			\chi(a_1 b) 
			x_{(b a_2)}
		\right\}_{a \in \R^{2}}
	\right)
\]

\section{Main Results}\label{Sec:MainResults}

We write $S_{n}$ for the symmetric group 
acting on the set 
$\{1,2,\dots,n\}$, 
equipped with the composition of permutations. 
For any $\R$-linear code~$C$, the code 
$C^{\sigma}:= \{\bm{u}^{\sigma} \mid \bm{u}\in C\}$ 
for any permutation~$\sigma \in S_{n}$ 
is called \emph{permutationally equivalent} to $C$, 
where 
$\bm{u}^{\sigma} := (u_{\sigma(1)},\dots, u_{\sigma(n)})$.

\begin{df}
	Let $C$ be an $\R$-linear code, 
	and $\bm{w} \in \R^{n}$.
	Then the \emph{average Jacobi polynomial}
	of~$C$ 
	with respect to $\bm{w} \in \R^{n}$
	is defined as follows:
	\[
		Jac^{av}(C,\bm{w}; \{x_{a}\}_{a \in \R^{2}}) 
		:= 
		\dfrac{1}{n!}
		\sum_{\sigma \in S_{n}} 
		Jac(C^{\sigma},\bm{w}; \{x_{a}\}_{a \in \R^{2}}).
	\]
\end{df}

Clearly, we have the MacWilliams identity 
for the average Jacobi polynomial as follows:
\[
	Jac^{av}
	(C^{\perp},\bm{w}; \{x_{a}\}_{a \in \R^{2}})\\
	=
	\dfrac{1}{|C|} 
	Jac^{av}
	\left(
		C,\bm{w};
		\left\{
			\sum_{b \in \R}
			\chi(a_1 b) 
			x_{(b a_2)}
		\right\}_{a \in \R^{2}}
	\right).
\]

Now we have the following result.
We will prove the following theorem
in a more general setting 
in Theorem~\ref{ThAvJointJacobi}.
Before stating the theorem we put
\[
	\dbinom{n}{n_{1},\ldots,n_{k}} 
	:=
	\dfrac{n!}{n_{1}!\cdots n_{k}!}.
\]

\begin{thm}\label{ThAverageJacobi}
	Let $C$ be an $\R$-linear code of length $n$,
	and $\bm{w} \in \R^{n}$. 
	Again let 
	$L$ be the composition of $n$
	and
	$R$ be the Jacobi composition of $n$
	such that
	\begin{align*}
	L 
	&=
	\left(
	\sum_{b \in \R}
	R_{(\omega_{0} b)},
	\ldots,
	\sum_{b \in \R}
	R_{(\omega_{|\R|-1} b)}
	\right),
	\\
	\ell(\bm{w}) 
	&=
	\left(
	\sum_{b \in \R}
	R_{(b\omega_{0})},
	\ldots,
	\sum_{b \in \R}
	R_{(b\omega_{|\R|-1})}
	\right).
	\end{align*}
	Then
	\begin{align*}
	Jac^{av}
	(C,& \bm{w}; \{x_{a}\}_{a \in \R^{2}})\\
	& = 
	\sum_{L,R} 
	A_{L}^{C} 
	\dfrac{
		\prod\limits_{b \in \R}
		\dbinom{\ell_{b}(\bm{w})}
		{R_{(\omega_{0} b)},
			\ldots,
			R_{(\omega_{|\R|-1} b)}}}
	{\dbinom{n}{L_{\omega_{0}},\ldots,L_{\omega_{|\R|-1}}}}
	\prod_{a \in \R^{2}} 
	x_{a}^{R_{a}}.
	\end{align*}
\end{thm}

\begin{df}
	Let $C$ and $D$ be two $\R$-linear codes of length~$n$.
	Then
	the \emph{average complete joint Jacobi polynomial} 
	of codes~$C$ and $D$ 
	with respect to $\bm{w} \in \R^{n}$ 
	is defined as follows:
	\[
		\Jac^{av}
		(C,D,\bm{w}; \{x_{a}\}_{a \in \R^{3}}) 
		:= 
		\dfrac{1}{n!}
		\sum_{\sigma \in S_{n}} 
		\Jac(C^{\sigma},D,\bm{w}; \{x_{a}\}_{a \in \R^{3}}).
	\]
\end{df}

\begin{rem}
	We have the following remarks.
	\begin{enumerate}
		\item 
		$\Jac^{av}(C,D,\bm{w}; \{x_{a}\}_{a \in \R^{3}}) 
		\neq 
		\Jac^{av}(D,C,\bm{w}; \{x_{a}\}_{a \in \R^{3}})$.
		\item  
		There exists some $\sigma \in S_{n}$ such that 
		\[
			\Jac^{av}(C^{\sigma},D^{\sigma}, \bm{w}; \{x_{a}\}_{a \in \R^{3}})
			\neq 
			\Jac^{av}(C,D, \bm{w}; \{x_{a}\}_{a \in \R^{3}}).
		\]
		\item 
		If $D = \{(0,0,\ldots,0)\}$, 
		then
		\[
			\Jac^{av}(C,D,\bm{w}; \{x_{a}\}_{a \in \R^{3}})
			=
			Jac^{av}(C,\bm{w};\{x_{b}\}_{b \in \R^{2}}).
		\]
	\end{enumerate}
\end{rem}

From Theorem~\ref{ThMacWilliams}, 
we have the MacWilliams identity 
for the average joint Jacobi polynomial as follows:
\begin{itemize}
	\item [(i)]
	$\begin{aligned}[t]
		\Jac^{av}
		(C,D^{\perp},&\bm{w}; \{x_{a}\}_{a \in \R^{3}})\\
		= &
		\dfrac{1}{|D|} 
		\Jac^{av}
		\left(
			C,D,\bm{w};
			\left\{
				\sum_{b \in \R}
				\chi(a_2 b) x_{(a_1 b a_3)}
			\right\}_{a \in \R^{3}}
		\right).
	\end{aligned}$
	\item [(ii)]
	$\begin{aligned}[t]
		\Jac^{av} 
		(C^{\perp},D,&\bm{w}; \{x_{a}\}_{a \in \R^{3}})\\
		= &
		\dfrac{1}{|C|} 
		\Jac^{av}
		\left(
			C,D,\bm{w}; 
			\left\{
				\sum_{b \in \R}
				\chi(a_1 b) x_{(b a_2 a_3)}
			\right\}_{a \in \R^{3}}
		\right).
	\end{aligned}$
	\item [(iii)]
	$\begin{aligned}[t]
		&\Jac^{av} 
		(C^{\perp},D^{\perp},\bm{w}; \{x_{a}\}_{a \in \R^{3}})\\
		& =
		\dfrac{1}{|C||D|} 
		\Jac^{av}
		\left(
			C,D,\bm{w}; 
			\left\{
				\sum_{b_1,b_2 \in \R}
				\chi(a_1 b_1+a_2 b_2) 
				x_{(b_1 b_2 a_3)}
			\right\}_{a \in \R^{3}}
		\right).
	\end{aligned}$
\end{itemize}

Now we have the following result analogue to the main theorem in~\cite{Y1989}
which represents the average of the complete joint Jacobi polynomials
of two codes $C$ and $D$ of length $n$
with respect to $\bm{w} \in \R^{n}$ 
by using 
the compositions of $n$ and its distribution in 
the codes. 

\begin{thm}\label{ThAvJointJacobi}
	Let $C$ and $D$ be two $\R$-linear codes of length $n$,
	and $\bm{w} \in \R^{n}$.
	Let $L$ be the composition of $n$ 
	and $R$ be the Jacobi composition of $n$. 
	Again let $H$ be the joint Jacobi composition of $n$ 
	such that
	\begin{align*}
		L &=
		\left(
		\sum_{a = (a_{1},a_{2}) \in \R^{2}}
		H_{(b a_{1} a_{2})} : b \in \R
		\right),
		\\
		R &=
		\left(
		\sum_{b \in \R}
		H_{(b a_{1} a_{2})} 
		: 
		a = (a_{1}, a_{2}) \in \R^{2}
		\right).
	\end{align*}
	Then
	\begin{align*}
		\Jac^{av}
		(C,D,& \bm{w}; \{x_{c}\}_{c \in \R^{3}})\\
		& = 
		\sum_{L,R,H} 
		A_{L}^{C} 
		B_{R}^{D,\bm{w}}
		\dfrac{
		\prod\limits_{a \in \R^{2}}
		\dbinom{R_{a}}
		{H_{(\omega_{0} a_{1} a_{2})},
		\ldots,
		H_{(\omega_{|\R|-1} a_{1} a_{2})}}}
		{\dbinom{n}{L_{\omega_{0}},\ldots,L_{\omega_{|\R|-1}}}}
		\prod_{c \in \R^{3}} 
		x_{c}^{H_{c}}.
	\end{align*}
\end{thm}

\begin{proof}
	Let $C$ and $D$ be two $\R$-linear codes of length~$n$. 
	Then the joint Jacobi polynomial of $C$ and $D$ 
	with respect to~$\bm{w}\in \R^{n}$ is
\begin{equation}\label{EquDefCJWE}
	\Jac(C,D,\bm{w};\{x_{c}\}_{c \in \R^{3}}) 
	= 
	\sum_{H} B_{H}^{C,D,\bm{w}} 
	\prod_{c \in \R^{3}}
	x_{c}^{H_{c}},
\end{equation}
where $	\sum_{c \in \R^{3}} H_{c} = n$. 
Now define
\begin{multline*}
	N_{L,R,H}^{C,D,\bm{w}} 
	:= \\
	\#
	\{
	(\bm{u},\bm{v}) \in C \times D 
	\mid 
	\comp(\bm{u}) = L, 
 	\comp_{\bm{w}}(\bm{v}) = R,
	\comp_{\bm{w}}(\bm{u},\bm{v}) = H
	\}.
\end{multline*}
Therefore, 
$B_{H}^{C,D,\bm{w}} = N_{L,R,H}^{C,D,\bm{w}}$,	
for 
\begin{align*}
	L &=
	\left(
	\sum_{a = (a_{1},a_{2}) \in \R^{2}}
	H_{(b a_{1} a_{2})} : b \in \R
	\right),
	\\
	R &=
	\left(
	\sum_{b \in \R}
	H_{(b a_{1} a_{2})} 
	: 
	a = (a_{1}, a_{2}) \in \R^{2}
	\right).
\end{align*}
Hence we can write from Equation~(\ref{EquDefCJWE})
\begin{equation*}
	\Jac(C,D,\bm{w};\{x_{c}\}_{c \in \R^{3}}) 
	= 
	\sum_{L,R,H} 
	N_{L,R,H}^{C,D,\bm{w}} 
	\prod_{c \in \R^{3}}
	x_{c}^{H_{c}}.
\end{equation*}
Now
\begin{align*}
	\sum_{\sigma \in S_n} 
	N_{L,R,H}^{C^\sigma,D,\bm{w}} 
	& = 
	\# 
	\{
		(\bm{u},\bm{v},\sigma) 
		\in T_{L}^{C} \times T_{R}^{D,\bm{w}} \times S_{n} 
		\mid
		\comp_{\bm{w}}(\bm{u}^\sigma,\bm{v}) = H
	\} \\		
	& = 
	\sum_{\bm{u} \in T_{L}^{C}}
	\sum_{\bm{v} \in T_{R}^{D,\bm{w}}} 
	\# 
	\{
	\sigma \in S_{n} 
	\mid 
	\comp_{\bm{w}}(\bm{u}^\sigma,\bm{v}) 
	= H
	\}.
\end{align*}
We observe that 
the order of a subgroup of $S_{n}$ which 
stabilizes 
$\bm{u} \in T_{L}^{C}$ 
is 
$\prod_{b \in \R }L_{b}!$. 
Therefore
\begin{align*}
	\sum_{\sigma \in S_{n}} 
	N_{L,R,H}^{C^\sigma,D,\bm{w}} 
	& = 
	\sum_{\bm{u} \in T_{L}^{C}}
	\sum_{\bm{v} \in T_{R}^{D,\bm{w}}} 
	\prod_{b \in \R }
	L_{b}! M_{L,H},
\end{align*}
where
\[
	M_{L,H}
	:=
	\#
	\{
	\bm{u}^{\prime} \in \R^{n} 
	\mid 
	\comp(\bm{u}^{\prime}) = L,
	\comp_{\bm{w}}(\bm{u}^{\prime},\bm{v}) = H
	\}.
\]
Therefore
\begin{align*}
	\sum_{\sigma \in S_{n}} 
	N_{L,R,H}^{C^\sigma,D,\bm{w}} 
	& =		
	\sum_{\bm{u} \in T_{L}^{C}}
	\sum_{\bm{v} \in T_{R}^{D,\bm{w}}} 
	\prod_{b \in \R}
	L_{b}! 
	\prod_{a \in \R^{2}}
	\dfrac{R_{a}!}
	{\prod_{b \in \R}
	H_{(b a_{1} a_{2})}!} \\		
	& = 
	A_{L}^{C} B_{R}^{D,\bm{w}}
	\prod_{b \in \R}
	L_{b}! 
	\prod_{a \in \R^{2}}
	\dfrac{R_{a}!}
	{\prod_{i=0}^{|\R|-1}
	H_{(\omega_{i} a_{1} a_{2})}!} \\
	& = 
	A_{L}^{C} B_{R}^{D,\bm{w}}
	n!
	\dfrac
	{\prod_{a \in \R^{2}}
	\dfrac{R_{a}!}
	{\prod_{b \in \R}
	H_{(b a_{1} a_{2})}!}}
	{\dfrac{n!}{\prod_{b \in \R}L_{b}!}}\\
	&= 
	A_{L}^{C} B_{R}^{D,\bm{w}} 
	n!
	\dfrac{
	\prod\limits_{a \in \R^{2}}
	\dbinom{R_{a}}
	{H_{(\omega_{0} a_{1} a_{2})},
	\ldots,
	H_{(\omega_{|\R|-1} a_{1} a_{2})}}}
	{\dbinom{n}{L_{\omega_{0}},\ldots,L_{\omega_{|\R|-1}}}}.
\end{align*}
Now we have
\begin{align*}
	\Jac^{av}
	& (C,D,\bm{w}; \{x_{c}\}_{c \in \R^{3}})\\ 
	& = 
	\dfrac{1}{n!} 
	\sum_{\sigma\in S_{n}} 
	\Jac(C^\sigma,D,\bm{w};x_{c}: c \in \R^{3})\\
	&= 
	\dfrac{1}{n!} 
	\sum_{L,R,H} 
	\sum_{\sigma\in S_{n}} 
	N_{L,R,H}^{C^\sigma,D} 
	\prod_{c \in \R^{3}}
	x_{c}^{H_{c}} \\
	&= 
	\sum_{L,R,H} 
	A_{L}^{C} B_{R}^{D,\bm{w}}
	\dfrac{
	\prod\limits_{a \in \R^{2}}
	\dbinom{R_{a}}
	{H_{(\omega_{0} a_{1} a_{2})},
	\ldots,
	H_{(\omega_{|\R|-1} a_{1} a_{2})}}}
	{\dbinom{n}{L_{\omega_{0}},\ldots,L_{\omega_{|\R|-1}}}}
	\prod_{c \in \R^{3}}
	x_{c}^{H_{c}}.
\end{align*}
This completes the proof.
\end{proof}

\section{Average of $(g+1)$-fold Complete Joint Jacobi Polynomials}
\label{Sec:AvHigherGenJacobi}

In this section,
we present a generalization of the average complete joint Jacobi polynomials
of linear codes over~$\R$.
We call these Jacobi polynomials the average $(g+1)$-fold complete joint Jacobi
polynomials.
We obtain a generalized MacWilliams identity for these Jacobi polynomials.
We also discuss an analogy of Theorem~\ref{ThAvJointJacobi} for the 
average $(g+1)$-fold complete joint Jacobi
polynomials. 

Let $\bm{w} \in \R^{n}$.
Then we denote the 
$g$-\emph{fold joint Jacobi composition} 
of~$\bm{u}_{1},\ldots,\bm{u}_{g} \in \R^{n}$
with respect to $\bm{w}$ by
\begin{align*}
	\comp_{\bm{w}}(\bm{u}_{1},\ldots,\bm{u}_{g})
	& :=
	h(\bm{u}_{1},\ldots,\bm{u}_{g};\bm{w})\\
	& :=
	(h_{a}(\bm{u}_{1},\ldots,\bm{u}_{g};\bm{w}):a\in\R^{g+1}),
\end{align*}
where
$h_{a}(\bm{u}_{1},\ldots,\bm{u}_{g};\bm{w})$ 
denotes the number of coordinate position $i$ such that
$a = (u_{1i},\ldots,u_{gi},w_{i})$.
It is immediate that
\[
	\sum_{a \in \R^{g+1}}
	h_{a}(\bm{u}_{1},\ldots,\bm{u}_{g};\bm{w})
	=
	n.
\]
Now we define the $g$-\emph{fold joint Jacobi composition} $H^{(g)}$ of $n$
by $H^{(g)} = (H_{a}^{(g)}: a \in \R^{g+1})$,
where the non-negative integer components 
$H_{a}^{(g)}$ satisfy the following condition:
\[
	\sum_{a \in \R^{g+1}}
	H_{a}^{(g)}
	= 
	n.
\]
\begin{df}
	Let $C_{1}, C_{2}, \ldots, C_{g}$ 
	be $\R$-linear codes of length~$n$.
	Then the $g$-\emph{fold complete joint Jacobi polynomial}
	of~$C_{1}, C_{2}, \ldots, C_{g}$ with respect to $\bm{w} \in \R^{n}$
	is defined as follows:
	\begin{align*}
		\Jac
		(C_{1},\ldots,C_{g},\bm{w};\{x_{a}\}_{a \in \R^{g+1}})
		:=
		& \sum_{\bm{u}_{1}\in C_{1},\ldots,\bm{u}_{g}\in C_{g}}
		\prod_{a \in \R^{g+1}}
		x_{a}^{h_{a}(\bm{u}_{1},\ldots,\bm{u}_{g},\bm{w})}\\
		=
		& \sum_{H^{(g)}}
		B_{H^{(g)}}^{C_{1},\ldots,C_{g},\bm{w}}
		\prod_{a \in \R^{g+1}}
		x_{a}^{H_{a}^{(g)}},
	\end{align*}
	where
	$B_{H^{(g)}}^{C_{1},\ldots,C_{g},\bm{w}}$
	is the number of $g$-tuple 
	$(\bm{u}_{1},\ldots,\bm{u}_{g}) \in C_{1}\times\cdots\times C_{g}$
	such that
	$\comp_{\bm{w}}(\bm{u}_{1},\ldots,\bm{u}_{g}) = H^{(g)}$.
\end{df} 

The MacWilliams identity for the $g$-fold joint weight enumerators of codes
over~$\ZZ_{k}$ was given in~\cite{DHO}. 
Recently the MacWilliams identity for the $g$-th Jacobi polynomial 
of a binary code was given in~\cite{HOO2020}. 
Now we give a generalized MacWilliams identity for 
the $g$-fold complete joint Jacobi polynomials of $\R$-linear codes. 
Let $\widetilde{C}_{i}$ denote either $C_{i}$ or $C_{i}^{\perp}$.
Then
\[
	\delta(C_{i},\widetilde{C}_{i})
	:=
	\begin{cases}
		0 & \mbox{if } \widetilde{C}_{i} = C_{i},\\
		1 & \mbox{if } \widetilde{C}_{i} = C_{i}^{\perp}.
	\end{cases}
\]  
We recall the character~$\chi$ from Section~\ref{Sec:MacWilliams}.
Let $T_{\R} = (\chi(ab))_{a,b\in\R}$ be an~$|\R|\times |\R|$ matrix.

\begin{thm}[Generalized MacWilliams identity]\label{Thm:GenMacWilliams}
	Let $C_{1},\ldots,C_{g}$ 
	be $\R$-linear codes of length~$n$,
	and~$\Jac(C_{1},\ldots,C_{g},\bm{w};\{x_{a}\}_{a \in \R^{g+1}})$
	be the $g$-fold complete joint Jacobi polynomial for 
	$C_{1},\ldots,C_{g}$ 
	with respect to $\bm{w} \in \R^{n}$.
	Then
	\begin{multline*}
		\Jac
		(\widetilde{C}_{1},\ldots,\widetilde{C}_{g},\bm{w};\{x_{a}\}_{a \in \R^{g+1}})\\
		=
		\dfrac{1}{|C_{1}|^{\delta(C_{1},\widetilde{C}_{1})}\cdots|C_{g}|^{\delta(C_{g},\widetilde{C}_{g})}}
		T_{\R}^{\delta(C_{1},\widetilde{C}_{1})}
		\otimes
		\cdots
		\otimes
		T_{\R}^{\delta(C_{g},\widetilde{C}_{g})}
		\otimes
		T_{\R}^{0}\\
		\Jac(C_{1},\ldots,C_{g},\bm{w};\{x_{a}\}_{a \in \R^{g+1}}),
	\end{multline*}
	where $T_{\R}^{0}$ denotes the identity matrix~$I$ of order~$|\R|$.
\end{thm}

\begin{proof}
	Let 
	$\widetilde{C}_{k}$ be $C_{k}^{\perp}$ and 
	$\widetilde{C}_{i}$ be $C_{i}$ for $k \neq i$. 
	Then it is sufficient to show that
	\begin{multline*}
		|C_{k}|
		\Jac
		(C_{1},\ldots,C_{k-1},C_{k}^{\perp},C_{k+1},\ldots,C_{g},\bm{w};
		\{x_{a}\}_{a \in \R^{g+1}})\\
		=
		\underset{k-1}{\underbrace{I \otimes\cdots\otimes I}}
		\otimes
		\underset{k\text{-th}}{T_{\R}}
		\otimes
		\underset{g-k}{\underbrace{I \otimes\cdots\otimes I}}
		\otimes
		\underset{(g+1)\text{-th}}{I}\\
		\Jac
		(C_{1},\ldots,C_{k-1},C_{k},C_{k+1},\ldots,C_{g},\bm{w};
		\{x_{a}\}_{a \in \R^{g+1}}).
	\end{multline*}
	The proof is straightforward. So, we leave it for the readers.
\end{proof}

\begin{df}
	Let $C$, $D_{1},\ldots,D_{g}$ be $(g+1)$ $\R$-linear codes of length~$n$.
	Then
	the \emph{average $(g+1)$-fold complete joint Jacobi polynomial} 
	of codes $C$, $D_{1},\ldots,D_{g}$ 
	with respect to $\bm{w} \in \R^{n}$ 
	is defined as follows:
	\begin{multline*}
		\Jac^{av}(C,D_{1},\ldots,D_{g},\bm{w}; \{x_{a}\}_{a \in \R^{g+1}})\\ 
		:= 
		\dfrac{1}{n!}
		\sum_{\sigma \in S_{n}} 
		\Jac(C^{\sigma},D_{1},\ldots,D_{g},\bm{w}; \{x_{a}\}_{a \in \R^{g+1}}).
	\end{multline*}
\end{df}

It is immediate from Theorem~\ref{Thm:GenMacWilliams} 
that the average $(g+1)$-fold complete joint Jacobi polynomials of 
$\R$-linear codes
$C,D_{1},\ldots,D_{g}$ satisfy the following MacWilliams type identity: 
\begin{multline*}
	\Jac^{av}
	(\widetilde{C},\widetilde{D}_{1},\ldots,\widetilde{D}_{g},\bm{w};
	\{x_{a}\}_{a \in \R^{g+2}})\\
	=
	\dfrac{1}{|C|^{\delta(C,\widetilde{C})}|
		D_{1}|^{\delta(D_{1},\widetilde{D}_{1})}
		\cdots
		|D_{g}|^{\delta(D_{g},\widetilde{D}_{g})}}\\
	T_{\R}^{\delta(C,\widetilde{C})}
	\otimes
	T_{\R}^{\delta(D_{1},\widetilde{D}_{1})}
	\otimes
	\cdots
	\otimes
	T_{\R}^{\delta(D_{g},\widetilde{D}_{g})}
	\otimes
	T_{\R}^{0}\\
	\Jac^{av}(C,D_{1},\ldots,D_{g},\bm{w};\{x_{a}\}_{a \in \R^{g+2}}),
\end{multline*}
where 
notations carry the same meaning as in Theorem~\ref{Thm:GenMacWilliams}.

Now we have the following generalization of Theorem~\ref{ThAvJointJacobi}.
The proof of the following theorem is not so difficult.
So, we leave it for the readers.

\begin{thm}\label{Thm:GenJacYoshida}
	Let 
	$C, D_{1},\ldots,D_{g}$
	be $(g+1)$ linear codes of length~$n$ over~$\R$, 
	and $\bm{w} \in \R^{n}$.
	Let $L$
	be the composition of~$n$,
	where
	$L = (L_{a}:a\in\R)$,
	and
	$R_{1},\ldots,R_{g}$
	be the Jacobi compositions of~$n$,
	where
	$R_{i} = (R_{ib}:b\in\R^{2})$
	for $i = 1,\ldots,g$.
	Let 
	$H^{(g)}$ be the~$g$-fold joint Jacobi composition
	of~$n$, where
	$H^{(g)} = (H_{c}^{(g)} : c = (c_{1},\ldots,c_{g},c_{g+1}) \in \R^{g+1})$
	such that for~$i = 1,\ldots,g$,
	\begin{align*}
		R_{ib} 
		=
		\sum_{\widehat{c} = (c_{1},\ldots,c_{i-1},b_{1},c_{i+1},\ldots,c_{g},b_{2})\in\R^{g+1}}
		H_{\widehat{c}}^{(g)},
		\quad
		\mbox{where }
		b = (b_{1},b_{2}) \in \R^{2}.
	\end{align*}
	Again let
	$H^{(g+1)}$ 
	be the~$(g+1)$-fold joint Jacobi composition of~$n$,
	where~
	$H^{(g+1)} = (H_{d}^{(g+1)} : d = (d_{0},d_{1},\ldots,d_{g},d_{g+1}) \in \R^{g+2})$
	such that
	\begin{align*}
		L_{a}
		& =
		\sum_{\widehat{d} = (a,d_{1},\ldots,d_{g},d_{g+1}) \in \R^{g+2}}
		H_{\widehat{d}}^{(g+1)}
		\quad
		\mbox{for }
		a \in \R,\\
		H_{c}^{(g)}
		& =
		\sum_{\widehat{d} = (d_{0},c_{1},\ldots,c_{g},c_{g+1}) \in \R^{g+2}}
		H_{\widehat{d}}^{(g+1)}
		\quad
		\mbox{for }
		c = (c_{1},\ldots,c_{g},c_{g+1}) \in \R^{g+1}.
	\end{align*}
	Then
	\begin{multline*}
		\Jac^{av}
		(C,D_{1},\ldots,D_{g}, \bm{w}; \{x_{e}\}_{e \in \R^{g+2}})
		= 
		\sum_{L,H^{(g)},H^{(g+1)}} 
		A_{L}^{C} 
		B_{H^{(g)}}^{D_{1},\ldots,D_{g},\bm{w}}\\
		\dfrac{
			\prod\limits_{c = (c_{1},\ldots,c_{g},c_{g+1}) \in \R^{g+1}}
			\dbinom{H_{c}^{(g)}}
			{H_{(\omega_{0} c_{1}\cdots c_{g}c_{g+1})}^{(g+1)},
				\ldots,
				H_{(\omega_{|\R|-1} c_{1}\cdots c_{g}c_{g+1})}^{(g+1)}}}
		{\dbinom{n}{L_{\omega_{0}},\ldots,L_{\omega_{|\R|-1}}}}
		\prod_{d \in \R^{g+2}} 
		x_{d}^{H_{d}^{(g+1)}}.
	\end{multline*}
\end{thm}

\section{Average Jacobi Intersection Number}\label{Sec:AvJacobiInterNum}

The notion of the average intersection number was
introduced in~\cite{Y1989} for binary linear codes.
In~\cite{CMxxxx}, the average intersection number 
for $\R$-linear codes was studied. 
In this section, we give the notion of 
the average Jacobi intersection number 
for $\R$-linear codes. 

For $\bm{u},\bm{w} \in \R^{n}$,
we define
${\bm{u}}^{\bm{w}} = (u_{1}^{w_{1}},\ldots,u_{n}^{w_{n}})$
such that
\[
	u_{i}^{w_{i}}
	=
	\begin{cases}
		u_{i} & \mbox{ if } w_{i} = 0,\\
		0 & \mbox{ otherwise}.
	\end{cases}
\]
Let $C$ and $D$ be two 	$\R$-linear codes of length~$n$.
Then the \emph{Jacobi intersection} of the codes $C$ and $D$ 
with respect to $\bm{w} \in \R^{n}$ is defined as, 
\[
	C \cap_{\bm{w}} D 
	:=
	\{
	(\bm{u},\bm{v}) \in C \times D 
	\mid 
	\bm{u}^{\bm{w}} 
	= \bm{v}^{\bm{w}} 
	\}.
\]
Now we defined the \emph{average Jacobi intersection number} 
for $C$ and $D$ with respect to~$\bm{w}\in \R^{n}$ 
as follows:
\[
	\Delta^{\bm{w}}(C,D) 
	:= 
	\dfrac{1}{n!} \sum_{\sigma \in S_{n}} |C^\sigma \cap_{\bm{w}} D|.
\]
Clearly, 
$h_{a}(\bm{u},\bm{v};\bm{w}) = 0$
for
$a = (a_{1},a_{2},a_{3} : a_{1} \neq a_{2}, a_{3} = 0) \in \R^{3}$
if and only if~
$\bm{u}^{\bm{w}} = \bm{v}^{\bm{w}}$. 
Thus we have the following remark.

\begin{rem}\label{Rem:AvInterNum}
	For $a = (a_{1},a_{2},a_{3}) \in \R^{3}$, we let 
	\[
		y_{a}
		=
		\begin{cases}
			0 & \mbox{ if } a_{1} \neq a_{2} \text{ and } a_{3} = 0,\\
			1 & \mbox { otherwise}.
		\end{cases}
	\]
	Then 
	$\Jac^{av}(C,D,\bm{w}; \{x_{a} \leftarrow y_{a}\}_{a \in \R^{3}}) = \Delta^{\bm{w}}(C,D)$.
\end{rem}

Now we have the following result.

\begin{thm}\label{ThAvJacobiInterNum}
	Let $C$ and $D$ be two $\R$-linear codes of length~$n$
	and $\bm{w} \in \R^{n}$.
	Again let
	$R^{\prime}$ and $R^{\prime\prime}$
	be two Jacobi compositions of~$n$
	such that
	\[
		R^{\prime}_{(a\omega_{0})} 
		= 
		R^{\prime\prime}_{(a\omega_{0})}, 
		\text{ for } 
		a \in \R 
	\]
	and 
	\[
		\sum_{a \in \R}
		R^{\prime}_{(ab)}
		\leq 
		\ell_{b}(\bm{w}),
		\text{ for }
		b \in \R
	\]
	Let $L$ be the composition of~$n$ 
	such that
	\[
		L =
		\left(
		\sum_{a \in \R}
		R^{\prime}_{(\omega_{0} a)},
		\ldots,
		\sum_{a \in \R}
		R^{\prime}_{(\omega_{|\R|-1} a)}
		\right).
	\]
	Then
	\[
		\Delta^{\bm{w}}(C,D)
		=
		\sum_{R^{\prime},R^{\prime\prime},L} 
		A_{L}^{C} B_{R^{\prime\prime}}^{D,\bm{w}}
		\dfrac
		{\prod_{a \in \R}
		\dbinom{\ell_{a}(\bm{w})}
		{R^{\prime}_{(\omega_{0}a)},
		\ldots, 
		R^{\prime}_{(\omega_{|\R|-1}a)}}}
		{\dbinom{n}{L_{\omega_{0}},\ldots,L_{\omega_{|\R|-1}}}}.
	\]
\end{thm}

\begin{proof}
	Let $T_{R^{\prime}}^{C}$ 
	be the set of all elements of $C$ 
	with composition $L$ of $n$,
	and $T_{R^{\prime\prime}}^{D,\bm{w}}$
	be the set of all elements of $D$ 
	with Jacobi composition 
	$R^{\prime\prime}$ of $n$ 
	with respect to 
	$\bm{w} \in \R^{n}$. 
	Then we can write
	\begin{align*}
		n! \Delta^{\bm{w}}(C,D) 
		& = 
		\sum_{\sigma \in S_{n}} 
		|C^\sigma \cap_{\bm{w}} D|\\
		& = 
		\#\{(\bm{u},\bm{v},\sigma) \in C \times D \times S_{n} 
		\mid (\bm{u}^{\bm{w}})^\sigma = \bm{v}^{\bm{w}}\}\\
		& = 
		\sum_{R^{\prime\prime},L} 
		\sum_{\bm{u} \in T_{L}^{C}} 
		\sum_{\bm{v} \in T_{R^{\prime\prime}}^{D,\bm{w}}} 
		\#
		\{
			\sigma \in S_{n} 
			\mid 
			(\bm{u}^{\bm{w}})^\sigma 
			= 
			\bm{v}^{\bm{w}}
		\}\\
		& = 
		\sum_{R^{\prime},R^{\prime\prime},L} 
		A_{L}^{C} 
		B_{R^{\prime\prime}}^{D,\bm{w}}\\
		&
		\prod_{b \in \R} 
		L_{b}!
		\prod_{a \in \R} 
		\dfrac{\ell_{a}(\bm{w})!}
		{R^{\prime}_{(\omega_{0}a)}!
		\cdots 
		R^{\prime}_{(\omega_{|\R|-1}a)}!}.	
	\end{align*}
	Hence this completes the proof.
\end{proof}

\subsection*{Some Numerical Examples}
Here we give some examples of the average Jacobi intersection numbers for some Type II codes over $\FF_{2}$.   

\begin{itemize}
	\item[(1)] 
	Let $e_{8}$ be the extended Hamming code, 
	and
	$\bm{w} \in \FF_{2}^{8}$.
	\begin{itemize}
		\item[(i)] 
		If $\wt(\bm{w}) = 1$, 
		then $\Delta^{\bm{w}}(e_{8},e_{8})= 4.8$.
		
		\item[(ii)] 
		If $\wt(\bm{w}) = 2$, 
		then $\Delta^{\bm{w}}(e_{8},e_{8})= 6.4$.
		
		\item[(iii)] 
		If $\wt(\bm{w}) = 3$, 
		then $\Delta^{\bm{w}}(e_{8},e_{8})= 9.6$.	
	\end{itemize}
	\item[(2)] 
	Let 
	$\bm{w}\in\FF_{2}^{16}$. 
	\begin{itemize}
		\item[(i)] If $\wt(\bm{w}) = 1$, 
		then 
		\begin{align*}
			\Delta^{\bm{w}}(e_{8}^{2},e_{8}^{2})
			& \approx  
			5.90769230769, \\
		 	\Delta^{\bm{w}}(d_{16}^{+},d_{16}^{+})
		 	& \approx 
		 	5.90769230769, \\ 
			\Delta^{\bm{w}}(d_{16}^{+},e_{8}^{2})
			& \approx 
			5.90769230769.
		\end{align*} 
		\item[(ii)] 
		If $\wt(\bm{w}) = 2$, 
		then 
		\begin{align*}
			\Delta^{\bm{w}}(d_{16}^{+},d_{16}^{+})
			& \approx 
			7.87692307692, \\
			\Delta^{\bm{w}}(e_{8}^{2},e_{8}^{2})
			& \approx 
			7.87692307692, \\
			\Delta^{\bm{w}}(d_{16}^{+},e_{8}^{2})
			& \approx 
			7.87692307692.
		\end{align*}
		\item[(iii)] 
		If $\wt(\bm{w}) = 3$, 
		then 
		\begin{align*}
			\Delta^{\bm{w}}(d_{16}^{+},d_{16}^{+})
			& \approx
			11.8153846154, \\
			\Delta^{\bm{w}}(e_{8}^{2},e_{8}^{2})
			& \approx 
			11.8153846154, \\
			\Delta^{\bm{w}}(d_{16}^{+},e_{8}^{2})
			& \approx 
			11.8153846154.
		\end{align*}		
	\end{itemize}
	\item[(3)] 
	Let $g_{24}$ be the extended Golay code, 
	and 
	$\bm{w}\in \FF_{2}^{24}$.
	\begin{itemize}
		\item[(i)] 
		If $\wt(\bm{w}) = 1$, 
		then 
		\begin{align*}
			\Delta^{\bm{w}}(g_{24},g_{24}) 
			& \approx 
			6.02048106692, \\
			\Delta^{\bm{w}}(d_{24}^{+},d_{24}^{+})
			& \approx 
			6.08859978358, \\
			\Delta^{\bm{w}}(g_{24},d_{24}^{+})
			& \approx 
			5.94427244582.
		\end{align*}
		\item[(ii)] 
		If $\wt(\bm{w}) = 2$, 
		then 
		\begin{align*}
			\Delta^{\bm{w}}(d_{24}^{+},d_{24}^{+})
			& \approx 
			8.11813304477, \\
			\Delta^{\bm{w}}(g_{24},g_{24})
			& \approx 
			8.02730808923, \\
			\Delta^{\bm{w}}(g_{24},d_{24}^{+})
			& \approx 
			7.92569659443.
		\end{align*}
		\item[(ii)] 
		If $\wt(\bm{w}) = 3$, 
		then 
		\begin{align*}
			\Delta^{\bm{w}}(d_{24}^{+},d_{24}^{+})
			& \approx 
			12.1771995672, \\
			\Delta^{\bm{w}}(g_{24},g_{24})
			& \approx 
			12.0409962134, \\
			\Delta^{\bm{w}}(g_{24},d_{24}^{+})
			& \approx 
			11.8885448916.		
		\end{align*}
		\item[(iii)] 
		If $\wt(\bm{w}) = 4$, 
		then 
		\begin{align*}
			\Delta^{\bm{w}}(g_{24},g_{24})
			& \approx 
			20.0581090736.
		\end{align*}
		\item[(iv)] 
		If $\wt(\bm{w}) = 5$, 
		then 
		\begin{align*}
			\Delta^{\bm{w}}(g_{24},g_{24})
			& \approx 
			36.0720806541.
		\end{align*}
	\end{itemize}
\end{itemize}

A combinatorial $t$-$(n, k, \lambda)$ \emph{design} 
(or a $t$-design for short) 
is a pair $\D= (X,\mathcal{B})$, 
where $X$ is a set of~$n$ points, 
and $\mathcal{B}$ a collection of $k$-element subsets of $X$ 
called \emph{blocks}, with the property that 
any $t$ points are contained in precisely $\lambda$ blocks. 

Let
$X := \{1,2,\ldots,n\}$, and
$ \bm{u} = (u_1, \ldots, u_n) \in \FF_{q}^{n}$.
Then the \emph{support} of $\bm{u}$ 
is the set of indices of its nonzero coordinates: 
$\supp(\bm{u}) := \{i \mid u_{i} \neq 0\}$.
For an $\FF_{q}$-linear code $C$ of length~$n$, let us define
$\mathcal{B}(C_{w}) := \{\supp(\bm{u}) \mid \bm{u} \in C_{w}\}$,
where
$C_{w} := \{\bm{u} \in C \mid \wt(\bm{u}) = w\}$.
In general, $\mathcal{B}(C_{w})$ is a multi-set.
For a $\FF_{q}$-linear code $C$ of length~$n$,
we say that 
$C_{w}$
is a $t$-design if 
$(X,\mathcal{B}(C_{w}))$ 
is a $t$-design.
We say that an $\FF_{q}$-linear code $C$ of length~$n$ 
is~$t$-homogeneous if 
for every given nonzero weight~$w$, $C_{w}$ is a~$t$-design.

Observing the values of the average Jacobi intersection numbers
of some binary Type~II codes,
we have the following conjectures.

\begin{conj}\label{ConjAsym}
	Let $C$ and $D$ be two Type~II codes of length $n$ over $\FF_{2}$,
	and $\bm{w}\in\FF_{2}^{n}$.
	If $C$ and $D$ are $t$-homogeneous and $\wt(\bm{w}) \leq t$,
	then
	\[
	\lim_{n\to\infty}
	\Delta^{\bm{w}}(C,D)
	=
	\begin{cases}
	6 & \mbox{ if } \wt(\bm{w}) \leq 1, \\
	8 & \mbox{ if } \wt(\bm{w}) = 2,\\
	12 & \mbox{ if } \wt(\bm{w}) = 3\\
	20 & \mbox{ if } \wt(\bm{w}) = 4\\
	36 & \mbox{ if } \wt(\bm{w}) = 5.
	\end{cases}
	\]
\end{conj}

\begin{conj}\label{Conjdesign}
	If $C$ and $D$ are $t$-homogeneous over~$\FF_{q}$, 
	and $\wt(\bm{w})\leq t$ for $\bm{w} \in \FF_{q}^{n}$,
	then the average Jacobi intersection number
	of $C$ and $D$ with respect to $\bm{w}$ 
	can be uniquely determined.
\end{conj}

Dougherty defined the $g$-fold joint weight enumerators for 
codes over Frobenius rings, and gave a generalization
of the MacWilliams relation for the weight enumerators 
in~\cite{Dougherty2017}.
In future work, 
we will give a Frobenius-code analogue 
of the results in this paper. 
Motivated by the work done by 
Miezaki and Oura in~\cite{MO2019},
Chakraborty, Miezaki and Oura 
introduced the concept of the
average complete joint cycle index
and gave a relation between 
the average complete joint cycle index and 
the average complete joint weight enumerator of codes 
in~\cite{CMOxxxx}.
We will discuss a relation between Jacobi polynomials of codes and cycle indices in the subsequent papers.


\section*{Acknowledgements}
The authors thank Manabu Oura for helpful discussions. 
The authors would also like to thank the anonymous
reviewers for their beneficial comments on an earlier version of the manuscript.
The second named author is supported by JSPS KAKENHI (18K03217).

\end{document}